\documentclass[11pt]{amsart}
\usepackage{epsfig,overpic,comment}
\usepackage[usenames,dvipsnames,svgnames,table]{xcolor}
\usepackage[pagebackref,linktocpage=true,colorlinks=true,linkcolor=Blue,citecolor=BrickRed,urlcolor=RoyalBlue]{hyperref}
\usepackage[alphabetic,msc-links,abbrev]{amsrefs} 

\title[Regularity for nonlinear free boundary problems]{A geometric approach to regularity for nonlinear free boundary problems with finite Morse index}
\author{Aram L. Karakhanyan}
\address{School of Mathematics, The University of Edinburgh, Peter Tait Guthrie Road, EH9 3FD, Edinburgh, UK}
\email{ aram6k@gmail.com}

\usepackage{amsmath,amsfonts,amssymb,amsthm,upref}

\usepackage{subfigure}

\usepackage{hyperref}

\usepackage{enumitem}

\usepackage{verbatim}

\usepackage{skak} 

\usepackage{mathrsfs}
\usepackage{esint}

\hypersetup{
    bookmarks=true,         
    unicode=true,          
    pdftoolbar=true,        
    pdfmenubar=true,        
    pdffitwindow=false,     
    pdfstartview={FitH},    
    pdftitle={My title},    
    pdfauthor={Author},     
    pdfsubject={Subject},   
    pdfcreator={Creator},   
    pdfproducer={Producer}, 
    pdfkeywords={keyword1} {key2} {key3}, 
    pdfnewwindow=true,      
    colorlinks=true,       
    linkcolor=blue,          
    citecolor=blue,        
    filecolor=magenta,      
    urlcolor=cyan           
}

\renewenvironment{proof}[1][\proofname ]{{\noindent \bfseries #1. }}{\qed \bigskip }

\newcommand{\R}{{\mathbb R}}

\renewcommand{\P}{{\mathcal P}}

\renewcommand\L{\mathscr L}

\newcommand\po[1]{ \{ {#1} > 0 \}}

\newcommand\dist{\operatorname{dist}} 

\newcommand{\e}{\varepsilon}

\newcommand{\supp}{\operatorname{supp}}

\renewcommand{\div}{\mathrm{div}}

\newcommand{\Index}{\mathrm{Index}}

\renewcommand\H{\mathcal H}

\newcommand\p{\partial}
\newcommand\fb[1]{\p \{ {#1} > 0 \}}
\newcommand\fbr[1]{\p_{\rm red} \{ {#1} > 0 \}}

\def\Om{\Omega}
\def\na{\nabla}

\newcommand\I[1]{\chi_{\{#1>0\}}}  

\newtheorem{theorem}{Theorem}[section]

\newtheorem{cor}[theorem]{Corollary}

\newtheorem{defn}{Definition}[section]

\newtheorem{lem}[theorem]{Lemma}

\newtheorem{prop}[theorem]{Proposition}
\newtheorem{remark}[theorem]{Remark}

\newtheoremstyle{named}{}{}{\itshape}{}{\bfseries}{.}{.5em}{\thmnote{#3 }#1}
\theoremstyle{named}

\numberwithin{equation}{section}

\thanks{2000 Mathematics Subject Classification. Primary 35R35, 	35J60.
\\ Keywords: Free boundary regularity, Morse index, global solutions, unique continuation.}

\begin{document}
\maketitle

\baselineskip=13pt    

\begin{abstract}

Let $u$ be a weak solution of the free boundary 
problem 
$$\L u=\lambda_0 \H^1\with\fbr{u}, u\ge 0,$$  where 
$\L u=\div(g(\na u)\na u)$ is a quasilinear elliptic operator and $g(\xi)$ is a given  function of $\xi$ satisfying 
some  structural conditions.  We prove that the free boundary $\fb u$ 
is continuously differentiable in $\R^2$, provided that $\po u$ has locally finite connectivity. 
Moreover, we show that the free boundaries  of weak solutions with finite {\it Morse index} must have finite connectivity. 
The weak solutions are locally Lipschitz continuous and non-degenerate stationary points of the Alt-Caffarelli type functional $J[u]=\int_{\Om}F(\na u)+Q^2\I u$.

The full regularity of  the free boundary is not fully understood even for the {\it minimizers} of $J[u]$ in the 
simplest case $g(\xi)=|\xi|^{p-2}, p>1$, partly because the methods
from the classical case $p=2$ cannot be generalized to the full range of $p$. 
Our method, however,  is very geometric and works even for the {\it stationary points} of the functional $J[u]$ for 
a large class of nonlinearities $F$.

\end{abstract}

\section{Introduction}
In this paper we study the weak solutions of the free boundary problem
\begin{equation}\label{pde-0}
\left\{
\begin{array}{lll}
\L u=0 &\hbox{in}\ B_r(x_0),\\
|\na u|=\ell &\hbox{on}\ \fb u\cap B_r(x_0),\\
u\ge 0&\hbox{in} \ B_r(x_0),\\
\end{array}
\right.
\end{equation}
where $\L u=\div(\rho(\na u)\na u)$ is a quasilinear elliptic operator,  $B_{r}(x_0)=\{x\in \R^2 : |x-x_0|<r\}$,  $\rho(\xi)$ is a given function of $\xi\in \R^2$ subject to some 
standard structural conditions, $\ell$ is a constant. 
The  solutions of \eqref{pde-0} can be seen as stationary points of the functional 
\begin{equation}
J[u]=\int_\Omega F(|\na u|)+Q^2\I u, \quad u\in \mathcal A=\{v\in W^{1, F}(\Omega), v-u_0\in W^{1,F}(\Omega)\}, 
\end{equation}
where $W^{1, F}(\Omega)$ is the Orlicz-Sobolev space of function defined on a Lipschitz domain
$\Omega\subset \R^2$, $u_0\in W^{1, F}(\Omega)$ is a given boundary condition, 
$\chi_D$ is the characterisitic function of a set $D$, 
and $\ell$ is determined from the implicit relation  
\[F'(\ell)\ell-F(\ell)=Q^2,\]
where $Q>0$ determines  the Bernoulli condition along the free boundary.
One can write \eqref{pde-0} in a more concise form 

\begin{equation}\label{pde-weak}
\L u=\lambda_0\H^1\with\fbr u, 
\end{equation}
where $\fbr u$ is the reduced boundary of the set $\po u$, $\H^s$ is the $s-$dimensional Hausdorff measure, and $\lambda_0=F'(\ell)$. Note that the latter relation can be 
used to recover the free boundary condition $|\na u|=\ell$ from the equation $\L u=\lambda_0\H^1\with\fbr u$.
We give the precise definition below, which is  valid in $\R^n, n\ge 2.$

\begin{defn}\label{weak-def}
A function  $u$ is said to be a weak solution of 
$\L u=\lambda_0 \H^{n-1}  \with \fbr u$ in a domain $\Omega\subset \R^n$
if the following is satisfied:
\begin{itemize}
 \item [\rm{1)}] $u\in W^{1, F}(\Omega)$
is continuous and non-negative in $\Omega$ and $\L u=0$ in $\{u>0\}$,
\item [\rm{2)}] for every bounded subdomain $D\Subset\Omega $ there are constants $0<c_{\min}\leq C_{\max} <\infty$ depending only on $n, \ell, \rho, \dist(\partial D, \Omega)$ such that for
every ball $B_r(x)\subset D$ centered at  free boundary point $x\in \partial \{u>0\}$ the following inequalities hold
$$c_{\min}\leq \frac 1 {r^n}\int_{\partial B_r(x)}u\leq C_{\max},$$
\item[\rm{3)}] $\{u>0\}$ is a set of locally finite perimeter  and 
$\L u=\lambda_0\H^{n-1}\with\partial_{\rm{red}}\{u>0\}$ in the following sense:
for test function
$\zeta \in C_0^\infty(\Omega)$ the equality
$$ -\int_\Omega\rho(\na u)\na u\na\zeta=
\lambda_0\int_{\partial_{\rm{red}}\{u>0\}}\zeta d\H^{n-1},
$$
holds. Here $\partial_{\rm red}\{u>0\}$ is the reduced boundary of $\{u>0\}$, see 4.5.5 \cite{Federer} for definition.
\end{itemize}

\end{defn}

\begin{remark}\label{weak-props}
In \cite{Wolan-adv}  a "flatness implies regularity`` type result is proven
for weak solutions. 
If $x_0\in\partial_{\rm red}\{u>0\}$ then
the free boundary near $x_0$ is a smooth surface. Hence, the free boundary 
condition $|\na u|=\ell$ is satisfied in the classical sense. 
Furthermore, the weak solutions enjoy the following properties:

\begin{itemize}
\item[$\bf1^\circ$] $\fb u$ is a set of locally finite perimeter and $\partial_{\rm red}\{u>0\}$ is open relative to $\partial\{u>0\}$,
\item[$\bf2^\circ$] $\partial_{\rm red}\{u>0\}$
is smooth, and 
$\H^{n-1}(\partial\{u>0\}\setminus \partial_{\rm red}\{u>0\})=0$, 
\item[$\bf3^\circ$] 
the gradient is upper-semicontinuous, i.e.
\[\limsup_{\begin{subarray}{c}{}x\to x_0 \\ x\in\po{u}\end{subarray}}|\na u(x)|=\ell.\]

\end{itemize}
\end{remark}

Our first result states that 
the finitely connected  free boundaries of weak solutions are smooth in 
$\R^2$. To elucidate  our method  we   first choose to formulate the result for
$\L u=\div(|\na u|^{p-2}\na u)$. 
\begin{theorem}\label{thm-1}
Let $\L=\Delta_p$ be the $p-$laplacian, i.e. $\rho(\xi)=|\xi|^{p-2}, F(\xi)=|\xi|^p, 1<p<\infty$ and $u$
be a weak solution of \eqref{pde-0} in $B_1$, the unit ball centered at the origin,  in the sense of Definition \ref{weak-def}. Suppose that there is a constant $\theta>0$ such that 
\begin{equation}\label{Lebeg-0}
 \frac{|\{u=0\}\cap B_r(x)|}{|B_r(x)|}\geq \theta>0, \qquad \forall B_r(x)\subset B_1, x\in{ \fb u}\cap B_1.
 \end{equation}
 If $\fb u$ has finite connectivity then it is a finite collection of continuously differentiable curves in $B_1$.
\end{theorem}

One can show that the finite Morse index free boundaries  with nonnegative generalized curvature must be of  finite connectivity.

\begin{theorem} \label{thm-2}
 Let $u$ be as in Theorem \ref{thm-1}. If $p\ge 2$ and the  curvature measure $\kappa$ of $\fb u$ exists and it is nonnegative then the finite Morse index solutions have locally finitely connected free boundaries.  
\end{theorem}

The Morse index of some $x_0\in \fb u$ is defined in standard way, that is, it is the number $m$ of negative eigenvalues $\theta_1, \dots, \theta_m$ 
of the problem $-\div(A\na \phi)-\kappa\phi \H^1\with\fbr u=\theta |\na u|^{p-2} \phi$ in $\R^2\setminus \{x_0\}$
where $A_{ij}=|\na u|^{p-2}(\delta_{ij}+(p-2)u_iu_j|\na u|^{-2})$. This eigenvalue problem arises naturally when one computes the 
second variation of the energy \cite{CJK}, \cite{K18}. The zero Morse index solutions are the minimizers, see Section \ref{sec:3} for precise definitions.

\section{The curvature of free boundary}
The proof of Theorem \ref{thm-1} is a combination of two lemmas to follow. In the first one
we show that any component of $\fb u$ with positive $\H^1$ measure must be smooth.

\begin{lem}\label{lem-ddt}
Let $u$ be as in Theorem \ref{thm-1} and $\gamma\subset \fb u$ a component of free boundary such that 
$\H^1(\gamma)>0$. Then $\gamma$ is convex and smooth. 
\end{lem}

\begin{proof}
We employ a compactness argument to show 
that  at every point of $\gamma$ the blow-up limit $u_0$  exists and $u_0$ is a weak solution thanks to  condition \eqref{Lebeg-0}. 
Consider $u_k(x)=\frac{u(x_0+r_k x)}{r_k}$ for some positive sequence $r_k\downarrow 0$ with $x_0\in\gamma\subset \partial\{u>0\}$.
From  the Lipschitz continuity of $u$ \cite{DP-2D, Wolan}, 
it follows that $\{u_k\}$ is locally uniformly Lipschitz.
 By a customary compactness argument there exists a subsequence, still denoted  $\{u_k\}$,  
converging to a limit $u_0\in W^{1, \infty}_{\rm loc}(\R^2)$
such that
  \begin{eqnarray}\label{HD0}
   u_k\rightarrow u_0 \qquad \textrm{ strongly in} \ W^{1, p}_{\rm loc}(\R^2) \textrm{ and } C^\alpha_{\rm loc}(\R^n), \forall \alpha\in (0,1) \
\textrm{as}\ k\rightarrow \infty,\\\label{HD}
   \partial \{ u_k >0 \} \rightarrow \partial \{ u_0 >0 \}\qquad \textrm{in Hausdorff distance}\  d_\H\textrm { locally in}\   \R^2, \\\label{HD1}
   \I{u_k}\rightarrow \I{u_0}\qquad  \textrm{in}\  L^1_{\rm loc}(\R^2).
  \end{eqnarray}
For the proofs of \eqref{HD0}-\eqref{HD1} we refer the reader to 
\cite{DP-2D}, \cite{Wolan-adv}.

Let $u_0$ be a blow-up of $u$ at $x_0\in \gamma$, 
then by Proposition \ref{prop-closed} (see Appendix) $u_0$ is a weak solution. 
Therefore from Remark \ref{weak-props} we have 
$|\na u_0|=\ell$ on $\fbr {u_0}$ and 
\begin{equation}\label{eq:Carol}
|\na u_0(x)|\le \ell, \quad x\in \R^2.
\end{equation} 
Let $S\subset \fbr {u_0}$ be a smooth 
connected curve and $S'\subset \po{u_0}$ a smooth perturbation of $S$ such that 
$S$ and $S'$ have the same endpoints. Consider the domain $D\subset \po{u_0}$ bounded by 
$S$ and $S'$, i.e.  
$\partial D=S\cup S'$.
We have 
\begin{eqnarray*}
0&=&\int_{D}\L u_0=\int_{S}\rho(\na u_0)(\na u_0\cdot \nu)+\int_{S'}\rho(\na u_0)(\na u_0\cdot \nu)\\
&=&-\rho(\ell)\ell\H^1(S)+\int_{S'}\rho(\na u_0)(\na u_0\cdot \nu).
\end{eqnarray*}
Utilizing the estimate  $|\na u_0(x)|\le \ell, x\in \R^2$ (see Remark \ref{weak-props} ${\bf 3^\circ}$) we infer 
\begin{equation}\label{curv}
\H^1(S)\le \H^1(S').
\end{equation}
Since $S$ is smooth we can 
locally paramatrize it as  $x_2=h(x_1), x\in(-\delta, \delta)$ for suitable choice of coordinate axis $x_1, x_2.$
Suppose that $S'$ is given by $x_2=h(x_1-t\psi(x_1))$, where 
$t>0$ is  small and $0\le \psi\in C_0^\infty(-\delta, \delta)$. Then from  \eqref{curv} we have
\begin{eqnarray}\label{blya}
0&\ge& \frac1t\int_{-\delta}^\delta[\sqrt{1+(h')^2}-\sqrt{1+(h'-t\psi')^2}]=\\\nonumber
&=&\int_{-\delta}^\delta\frac{2h'\psi'-t(\psi')^2}{\sqrt{1+(h')^2}+\sqrt{1+(h'-t\psi')^2}}\to\ \ \hbox{as}\ t\to 0 \\\nonumber
&\to&\int_{-\delta}^\delta\frac{h'\psi'}{\sqrt{1+(h ')^2}}.
\end{eqnarray}
Therefore $\frac d{dx_1}\left(\frac{h'}{\sqrt{1+(h')^2}}\right)\ge 0$ on $(-\delta, \delta)$.
Consequently, the outer curvature $\kappa(S)\ge 0$, i.e. $S$ is a convex graph in $x_2$ direction.
Since $u_0$ is a weak solution  then it follows that  $\fb {u_0}\cap B_1$ is rectifiable and therefore 
$\fb {u_0}\cap B_1=\gamma_0\cup\left(\cup_{k=1}^\infty\gamma_k\right)$ such that 
$\gamma_k, k\ge 1$ are differentiable curves and $\H^1(\gamma_0)=0$. 
Moreover,  by \eqref{blya} every $\gamma_k, k\ge 1$ is convex.  
For some fixed $k_0$ let $y_0\in \partial \gamma_{k_0}$, the  relative boundary of $\gamma_{k_0}$.  Observe that $\gamma_{k_0}$ is convex so there is one sided sub-differential at $y_0$ (from the regular side of $\gamma_{k_0}$). If we blow-up 
 $u_{0}$ at $y_0$, then $u_{00}$ is again a weak solution thanks to Proposition \ref{prop-closed}.
 Moreover, at $0\in \fb{ u_{00}}$, the free boundary  contains a line on which $|\na u_{00}|=\ell$.
 Without loss of generality (because $\rho(\xi)$ is rotation invariant)
 we may assume that the positive semiaxis $x_1>0$ is a subset of the free boundary 
 $\fb{u_{00}}$ and $\L u_{00}=0$ in $\po{u_{00}}$.
 Continuing $u_{00}$ linearly across the positive semiaxis $x_1>0$ and letting 
 \[\widetilde u_{00}(x_1, x_2)=
 \left\{
 \begin{array}{lll}
 u_{00}(x_1, x_2) &\ \hbox{if}\ x_1\ge0, \\
 \ell x_1&\ \hbox{if}  \ x_1<0,   
 \end{array}
 \right.
 \] 
we see that $\L \widetilde u_{00}=0$ in some neighborhood of $x_1=0$.
Applying the unique continuation property of $\L$ \cite{GM} we infer that 
$u_{00}(x_1, x_2)=\ell x_1^+$. 
This implies that $u_0$ is flat at $y_0$ and the relative boundary of the convex arc $\gamma_{k_0}$ is empty. In other words, $\fb{u_0}$ is a smooth, convex, and complete curve in $\R^2$. This yields that $u$ is flat at $x_0$. Indeed,  since $\fb {u_0}$ is smooth then at $0\in \fb {u_0}$ we can take $\rho>0$,  small,  such that 
\[\fb{u_0}\cap B_\rho\subset \left\{ -\frac{\bar \sigma_0\rho}{2}<x\cdot e<\frac{\bar \sigma_0\rho}{2}   \right\}\cap B_\rho\]
where $e$ is the normal of $\fb{u_0}$ at $0$ and $\bar \sigma_0$ is the critical flatness constant, see Theorem 9.3 \cite{Wolan-adv}. In other words, $u_0$ belongs to the 
flatness class $F\left (\frac{\bar \sigma_0}2, 1; \infty\right)$ in $e$ direction.
Choose a sequence $r_k\to 0$ as above such that $u_k(x)=\frac{u(x_0+r_kx)}{r_k}\to u_0(x)$
and  \eqref{HD0}-\eqref{HD1} hold. In particular from 
\eqref{HD} it follows that 
\[\fb{u_k}\cap B_\rho\subset \left\{ -\frac{3\bar \sigma_0\rho}{4}<x\cdot e<\frac{3\bar \sigma_0\rho}{4}   \right\}\cap B_\rho\]
or $u_k\in F\left (\frac{3\bar \sigma_0}4, 1; \infty\right)$ for sufficiently large $k$. Applying Theorem 9.3 \cite{Wolan-adv} we get that $B_{\rho/4}\cap \fb {u_k}$ is $C^{1, \alpha}$ surface 
$\alpha>0$ in $e$ direction.
Consequently, pulling back to $u$ and using 
flatness implies regularity result (Theorem 9.4 \cite{Wolan-adv}) we infer that $\fb u$ is differentiable at $x_0$ and hence smooth. 
\end{proof}

Recall that if $K\subset \R^2$ is bounded convex set with non-empty interior then there is an ellipse 
$E\supset K$ of minimal area and center $x_0\in K$ so that 
\[
x_0+\frac12 (E-x_0)\subset K\subset E, 
\]
see \cite{deGuzman} Lemma 2.2 page 139.
\begin{cor}
Let $u$ be as in Theorem \ref{thm-1}. Let $\gamma$ be a closed convex component of $\fb u$ and $a\ge b$ are the semiaxis of John's ellipse
of the convex hull of $\gamma$.
Then there is universal $c_0$ such that 
\[
\frac a{b}\le c_0.
\] 
\end{cor}
\begin{proof}
Suppose the claim fails, then there are convex closed curves $\gamma_k\subset \fb u$ such that 
the semiaxis $a_k\ge b_k$ of John's ellipse $E_k$ satisfy $\frac{a_k}{b_k}\to \infty.$
Let $z_k$ be a free boundary point intersecting  the major semiaxis of $E_k$. Consider 
$u_k=\frac{u(z_k+a_kx)}{a_k}$. Then using a customary compactness argument (as in the proof of Lamma \ref{lem-ddt}) we can extract a 
subsequence such that $u_k$ converges to a weak solution $u_0$ for which \eqref{Lebeg-0} fails.
In view of Proposition \ref{prop-closed} this is a contradiction. 
\end{proof}

\section{Finite Morse index solutions}\label{sec:3}
In this section we prove Theorem \ref{thm-2}.  We show that the finite index free boundary of $u$ is "non-thinning", i.e. it is not possible to 
have $x_0\in \fb u$ and disjoint components $\gamma_k\subset \fb u$ such that $x_0\not \in\gamma_k, k=1, 2, \dots$ but for some  sequence $x_k\in \gamma_k$ we have $x_k\to x_0$.
Obviously, if $\fb u$ is thinning then we can assume that each $\gamma_k$ must be a closed convex curve. Indeed, if 
$u_0$ is a blow up of $u$ at $x_0$ then $\fb {u_0}$ is thinning too and by lemma \ref{lem-ddt}
each nontrivial bounded component of the free boundary is a closed convex smooth curve.

Recall that by Theorem 6.2 \cite{ACF-Q} $\Gamma$ is rectifiable. Then we define the 
rectifiable $1$ varifold $V= v(\Gamma,  \theta)$ (with multiplicity $\theta$) as the 
equivalence class of all pairs $(\widetilde\Gamma, \widetilde \theta)$ such that 
$\widetilde\Gamma$ is $\H^1$ rectifiable, $\H^1(\Gamma\triangle\widetilde \Gamma)=0$ and 
$\theta=\widetilde \theta$ a.e. on $\Gamma\cap \widetilde\Gamma$ \cite{Simon} page 77.
We say that $V$ has bounded first variation if there is a constant $c>0$ such that 
\begin{equation*}
\sup_{X\in C^{1}_0(U), |X|\le 1}\left|\int \div Xd\mu_V\right|\le c.
\end{equation*}

By the Riesz represenation theorem there is a vector measure  $H$ such that 
\begin{equation*}
\int \div X=-\int H\cdot X.
\end{equation*}
$\kappa:=|H|$ is called the curvature measure of the varifold $V$. If $\Gamma$ is smooth then $H$ coincides with the curvature vector.

\begin{defn}
Let $u$ be a weak solution and the curvature measure of $\fb u$ is locally finite. We say that $u$ has finite Morse index at $x_0\in \fb u$ if there is a constant $C_0>0$ such that 
\begin{eqnarray}\label{eq:Morse}
{\displaystyle
\dfrac{\int_{\po u}|\na u|^{p-2}\left(|\na \phi|^2+(p-2)\frac{(\na u\cdot \na \phi)^2}{|\na u|^2}\right)-\int_{\fb u} \kappa \phi^2}{\int_{\po u}|\na u|^{p-2}\phi^2}
\ge -C_0
}
\end{eqnarray}
whenever $\phi\in C_0^{0,1}(\R^2\setminus \{x_0\}).$
Here $\kappa$ is the curvature measure of the free boundary. If $x_0$ has finite index then we write $\Index (x_0)<\infty$.
\end{defn}

The second variation of the energy $J[u]$ is computed in \cite{K18}, where the following stability inequality for the minimizers had been  derived 
\begin{equation}
\int_{\fb u} \kappa \phi^2\le \int_{\po u}|\na u|^{p-2}\left(|\na \phi|^2+(p-2)\frac{(\na u\cdot \na \phi)^2}{|\na u|^2}\right)
\quad \forall \phi\in C_0^{0,1}(\R^2\setminus \{0\}).
\end{equation}
Thus finite Morse index means that there are finitely many negative eigenvalues if one minimizes the ratio in \eqref{eq:Morse}.

\begin{lem}
Let $u$ be as in Theorem \ref{thm-2}. Suppose $0\in \fb u$ and $\Index (0)<\infty$. Then $\fb u$ has finite connectivity near $0.$
\end{lem}
\begin{proof}
If there is a free boundary component joining $0$ with some nonzero point then 
the argument in the proof of Lemma \ref{lem-ddt} shows that $\fb u$ must be a smooth curve near $u$. 
Thus we assume that $\fb u$ is thinning near $0$. Let $r_k=2^{-k}$ and define 
\[
\xi_k(x)=\left\{
\begin{array}{ll}
0 & \textrm{in}\ B_{r_{k+1}},\\
1-\frac1{\log2}\log\frac{r_k}{|x|} & \textrm{in}\  B_{r_k}\setminus B_{r_{k+1}},\\
1 & \textrm{in}\  B_1\setminus B_{r_k},
\end{array}
\right.
\]
and,  set $\phi_k=\eta\xi_k$ where $\eta\in C_0^\infty(B_1)$ such that $0\le \eta\le1$
and $\eta=1$ in $B_{\frac12}$. Utilizing \eqref{eq:Morse} with $\phi=\phi_k$ gives 
\begin{eqnarray}\label{Morse-1}
\int_{\fb u} \kappa \phi_k^2
&\le&
 \int_{\po u}|\na u|^{p-2}\left(|\na \phi_k|^2+(p-2)\frac{(\na u\cdot \na \phi_k)^2}{|\na u|^2}\right)\\\nonumber
 &&+C_0\int_{\po u}|\na u|^{p-2}\phi^2_k\\\nonumber
&\le&
(p-1)\int_{\po u}|\na \phi_k|^2+C_0|B_1|,
\end{eqnarray}
where the last inequality follows from \eqref{eq:Carol}.
On the other hand we get 
\begin{eqnarray}
\int_{\po u}|\na \phi_k|^2
&=&
\int_{B_{r_k}\setminus B_{r_{k+1}}}|\na \phi_k|^2+\int_{B_{\frac12}\setminus B_{r_k}}|\na \phi_k|^2+\int_{B_1\setminus B_{\frac12}}|\na \phi_k|^2\\\nonumber 
&=&
\frac{2\pi }{(\log 2)^2}\int_{2^{-k-1}}^{2^{-k}}\frac{dt}{t}+\int_{B_1\setminus B_{\frac12}}|\na \eta|^2\\\nonumber
&\le&
\frac{2\pi }{\log 2}+C(\eta).
\end{eqnarray}
Hence, returning to \eqref{Morse-1} we get 
\begin{equation}\label{Morse-2}
\int_{\fb u} \kappa \phi_k^2\le \frac{2\pi }{\log 2}+C(\eta)+C_0|B_1|.
\end{equation}
If $\gamma_1, \dots, \gamma_N$ are closed convex components of $\fb u\cap B_{\frac12}$ then from \eqref{Morse-2} (with sufficiently large $k$) we get that 
\[
2\pi N\le \frac{2\pi }{\log 2}+C(\eta)+C_0|B_1|,
\]
where we used $\int_{\gamma_i}\kappa=2\pi$. Thus $N$ is finite and $\po u$ has finite connectivity at $0$. 
\end{proof}

\begin{remark}
There are examples of periodic weak solutions in $\R^2$ as a solution to a stationary problem in hydrodynamics  \cite{Baker},  \cite{Cro}.
In \cite{T} Traizet gave the full classification of smooth free boundaries  with finite connectivity in $\R^2$. In this respect 
Theorem 12 in \cite{T} says that if $\fb u$ has infinite connectivity then it cannot have finite Morse index. 
\end{remark}

\begin{remark}
It is well known that  a minimizer of $J_p[u]=\int_{B_1}|\na u|^p+\lambda\I u$  is also a
weak solution. 
For $\e>0$ small,  let $$ u_\e:=\max\{u-\e\zeta,\,0\},$$ 
where~$\zeta\in C_0^{0,1}(B_\rho(x_0))$, defined as 
\begin{equation*}
\zeta(x)=\left\{
\begin{array}{lll}
0 & \hbox{if}\ |x-x_0|>R,\\
\displaystyle\frac{\log(R/|x-x_0|)}{\log(R/r)}& \hbox{if}\ x\in B_R(x_0)\setminus B_r(x_0),\\
1 & \hbox{if}\ x\in B_r(x_0).
\end{array}
\right.
\end{equation*}
Then in $\R^2$ the comparison $J_p[u]\le J_p[u_\e]$ gives that 
every blow-up of $u$ must have constant gradient provided $p>2$, see \cite{AC}, \cite{ACF-Q}, \cite{DP-2D}.
Clearly, this argument cannot be used if $u$ is merely a stationary point. Moreover, even for the 
minimzers, it does not imply that the free boundary is continuously differentiable 
for the range $1<p<2$, cf \cite{DP-2D}. 
\end{remark}

In closing this section  we construct a sequence of weak solutions such that the measure 
theoretic boundary of its limit is empty, cf. \cite{AC} 5.8. This example shows that 
the condition \eqref{Lebeg-0} is necessary. 
Let us define  
\begin{equation*}
 u_\e(x)=\left\{\begin{array}{lll}
                 x_n-\e  &{\rm{for}}\  x_n\geq \e, \\
0 &{\rm{for}}\  |x_n|\leq \e, \\
\e-x_n & {\rm{for}}\  x_n\leq -\e.
                \end{array}
\right.
\end{equation*}
One can check that $u$ is a weak solution for every $\e>0.$
However, for $\e=0$ this is not true. In this case $\Delta_p u=2\H^{n-1}\with \fb{u}$ and 
$\H^{n-1}(\partial\{u>0\}\setminus \partial_{\rm{red}}\{u>0\})>0$ since $\partial_{\rm{red}}\{u>0\}=\emptyset$,
in other words the normal derivative $\partial_\nu u$ cannot be reconstructed from the free boundary 
data$\ell=2$.

\section{Generalizations}

 One can impose   various assumptions on $\rho$ to guarantee that the elliptic operator  
has nice properties. 
We formulate them in the following three hypotheses: 
\begin{itemize}
\item[{\bf (H1)}] $\L$ is a quasilinear elliptic operator such that the strong maximum principle,  interior $C^{1,\alpha}$ regularity theory, Harnack inequality for non-negative solutions are valid for the weak solutions of $\L u=0$.  Under these conditions it follows that if $u$ is a weak solution of \eqref{pde-0} then 
the gradient is upper semicontinuous, i.e. 
\begin{equation}\label{limsup}
\limsup_{\begin{subarray}{c}{}x\to x_0, \\ x\in\po{u}\end{subarray}}|\na u(x)|=\ell.
\end{equation}

\item[{\bf (H2)}]The  unique continuation  for the weak solutions of $\L u=0$ is valid, i.e.
if $\L u=0$ in $\Omega$ and there is subdomain $\Omega'\subset \Omega$ such that 
$u$ is affine on $\Omega'$ then $u$ is affine in $\Omega.$ 

\item[{\bf (H3)}] The class of weak solutions is closed with respect to blow-up.
\end{itemize}
If $\bf(H1)$-$\bf(H3)$ are valid then one can generalize the 
variational theory from \cite{AC}, \cite{ACF-Q} for a larger class of 
equations 
\begin{equation}\label{pde}
\L u=\div\left(g(|\na u|)\frac{\na u}{|\na u|}\right), \quad \rho(\xi)=\frac{g(|\xi|)}{|\xi|}
\end{equation}
where $g(t)=G'(t)$ and $G$ is a convex function satisfying $0<c < \frac{G'(t)}t < C<\infty$ with implicitly given free boundary conditions 
\[\Psi(\na u)=Q^2\]
where $\Psi$ is determined by $g$, see \cite{Wolan-adv}.

Some examples are as follows:  
\begin{itemize}
\item The classical Alt-Caffarelli functional $J_{\rm AC}[u]=\int_\Om |\na u|^2+Q^2\I u$ \cite{AC}
where $Q(x)>0, x\in \Om$ is a H\"older continuous function bounded away from zero and infinity.
\item The nonlinear version of $J_{\rm AC}$ 
\begin{equation}
\int_\Om F(|\na u|^2)+Q^2\I u 
\end{equation}
with $c_0\le F'(t)\le C_0, 0\le F''(t)\le C_0(1+t)^{-1}$ \cite{ACF-Q}. The weak solutions solve 
$\L u=\lambda_0\H^1\with\fbr u$ with implicitly defined free boundary condition 
\begin{equation}
\Phi(|\na u|^2)=Q^2, \quad \hbox{where}\quad \lambda_0=2\ell F'(\ell^2), |\na u|=\ell, \Phi(s)=2sF'(s)-F(s).
\end{equation}

\item The non-radially symmetric version of $J_{\rm ACF}$, namely 
\[\int_\Om f(\na u)+Q^2\I u\] 
under the assumtion that $p\cdot\na f(p)-f(p)$ is convex \cite{Weiss-95}.
\item The analogue of minimiziation problem for the functional $J_{\rm ACF}$ in the Orlicz-Sobolev 
spaces corresponding to the energy 
$\int_\Om G(|\na u|)+Q^2\I u$ where 
\[
\L u =\div\left(g(|\na u|)\frac{\na u}{|\na u|}\right), \quad g(s)=G'(s)
\]
and the free boundary condition is $|\na u|=\ell$ where $\ell$ is determined from the 
implicit relation  $G'(\ell)\ell-G(\ell)=Q^2$ \cite{Wolan-adv}. The weak equation is 
$\L u=\lambda_0\H^1\with\fbr{u}$ and 
$\lambda_0=g(\ell)$.
\item The $p(x)$-Laplacian model and the variable growth functional 
\[J_{p(\cdot)}[u]=\int_\Om\frac{|\na u|^{p(x)}}{p(x)}+Q^2\I u, \quad \ell:=\left(\frac{p(x)}{p(x)-1}Q^2\right)^{\frac1{p(x)}}\]
modelling the stationary flow of electrorheological fluids. 
In this case the free boundary condition is 
$\ell:=|\na u|$ and the differential operator $\L u:=\Delta_{p(x)}u=\div(|\na u|^{p(x)-2}\na u)$ 
with $\lambda_0=\ell^{p(x)-1}$, for constant case $p(x)=p$ see  \cite{DP-2D}.
\end{itemize}

In order to formulate the general result it is convenient to introduce the following classes $\mathcal P_r(x_0, M)$ of 
weak solutions: 
\begin{defn}
We say that $u\in \P_r(x_0, M)$ if 
\begin{itemize}
\item[(a)] $u\in C^{0, 1}(B_r(x_0))$ and $\sup_r|\na u|\le M$,
\item[(b)] $u\ge 0$ and $x_0\in \fb u$, 
\item[(c)] $u$ is a weak solution in $B_r(x_0)$.
\end{itemize}
We say that $u\in \P_r(x_0, M, \theta)$ if $u\in \P_r(x_0, M)$ and 
\[ \frac{|\{u=0\}\cap B_\rho(x)|}{|B_\rho(x)|}\geq \theta>0,\]
for any $B_\rho(x)\subset B_r(x_0), x\in \fb u$.
\end{defn}
Note that if $u\in P_r(x_0, M)$ and $u_s(x)=\frac{u(x_0+sx)}{s}$ then 
$u_s\in \P_{r/s}(0, M)$. Moreover, in view of Proposition \ref{prop-closed} (see Appendix)
if $u\in P_r(x_0, M, \theta)$ and $u_{s_k}\to u_0$ locally uniformly (for some $0<s_k\downarrow 0$) 
then $u_0\in \P_\infty(0, M, \theta)$. Therefore $\bf (H3)$ is valid for 
$\P_r(x_0, M, \theta)$. 

From the proof of Theorem \ref{thm-1} we have the following generalization.

\begin{theorem}
Let $u\in \P_1(x_0, M, \theta)$ such that 
the hypotheses {\bf (H1)-(H3)} are satisfied, then near $x_0$ the free 
boundary is a
continuously differentiable curve provided that $\po u$ has finite connectivity.  
\end{theorem}

\section{Appendix}
In this section we prove that $\P_r(x_0, M, \theta)$ is closed with respect to the blow-up procedure.
We choose to state the result in $\R^n$ and adapt the proof from \cite{Weiss}.

\begin{prop}\label{prop-closed}
Let  $u$ be a weak solution of $\L u=\lambda_0\H^{n-1}\with\partial\{u>0\}$ 
in the sense of Definition \ref{weak-def}. If $u\in \P_r(x_0,M, \theta)$ such that the blow-up  sequence $u_{k}(x)=\frac{u(x_0+\rho_kx)}{\rho_k}$ converges
locally uniformly to $u_0$, then $u_0\in \P_R(0, M, \theta)$ for any $R>0$. 
\end{prop}

\begin{proof} 
{\bf Step 1)}
Let $u_0$ be a blow-up limit, i.e.  let 
$u_k(x) =\frac{u(x_0+\rho_k x)}{\rho_k}$, $ x_0\in\partial\{u>0\}$ for some sequence 
$\rho_k\searrow 0, k\rightarrow \infty$. Then by a customary 
compactness argument \cite{Wolan-adv} Lemma 7.1 and Remark 8.2 we can extract a subsequence, 
still denoted by $\rho_k$, such that $u_k\to u$ locally uniformly. More precisely, we have that 
 \begin{eqnarray}\label{AHD0}
   u_k\rightarrow u_0 \qquad \textrm{ strongly in} \ W^{1, F}_{\rm loc}(\R^n) \textrm{ and } C^\alpha_{\rm loc}(\R^n), \forall \alpha\in (0,1) \
\textrm{as}\ k\rightarrow \infty,\\\label{AHD}
   \partial \{ u_k >0 \} \rightarrow \partial \{ u_0 >0 \}\qquad \textrm{in Hausdorff distance}\  d_\H\textrm { locally in}\   \R^n, \\\label{AHD1}
   \I{u_k}\rightarrow \I{u_0}\qquad  \textrm{in}\  L^1_{\rm loc}(\R^n).
  \end{eqnarray}
is a weak solution. Consequently, it follows that 
properties 1) and 2) in Definition \ref{weak-def} for $u_0$ 
hold true. Furthermore, 
$$\displaystyle\frac{|B_r(x))\cap \{u_0=0\}|}{|B_r(x)|}\geq \theta, \qquad x\in\partial \{u_0>0\},\qquad  B_r(x)\subset \R^n.$$

To show that $u_0$ is weak solution
it remains to verify the equation 3) in the Definition \ref{weak-def}.
We need to show two things: 
$\H^{n-1}(\partial\{u_0>0\}\setminus \partial_{\rm red}\{u_0>0\})=0$
and the smoothness of $\partial_{\rm red}\{u_0>0\}$ stated in 
the Remark \ref{weak-props} $\bf 1^\circ$-$\bf 3^\circ$.

{\bf Step 2)} Next,  we prove that $\{u_0>0\}$ is of finite perimeter. Take $\zeta(x)=\max(0,\min(1, \frac 1\e(R-|x|)) )$
in 3) of Definition \ref{weak-def} to conclude, after sending $\e$ to zero, that for a.e. $R>0$

\begin{eqnarray*}
 \H^{n-1}(\partial_{\rm red}\{u_k>0\}\cap B_{R}(0))&=&
\rho^{n-1}_k\H^{n-1}(\partial_{\rm red}\{u>0\}\cap B_{R\rho_k}(x_0))\\\nonumber
&=&\lambda_0
\int_{\partial B_{\rho_kR}(x_0)}g(|\na u|)\frac{\na u}{|\na u|}\cdot \nu
\H^{n-1}\\\nonumber
&\leq& C.
\end{eqnarray*}
From here the claim follows from the semi-continuity of the perimeter.

\smallskip 

Since the current boundary $T=\partial(\R^n\with \{u_0>0\}\cap B_R(0))$ is 
representable by integration, $\|T\|=\int_{B_R(0)} |D\I{u_0}|$, we get
from 4.5.6. (3) \cite{Federer} $\partial\{u_0>0\}\setminus \partial_{\rm red}\{u_0>0\}=K_0\cup K_+ $ where
$\H^{n-1}(K_+)=0$ and for 
$x_1\in K_0$,  $r^{1-n}\H^{n-1}(\partial_{\rm red}\{u_0>0\}\cap B_r(x_1))\rightarrow 0$ as $r\rightarrow 0$.

Let us show that $K_0=\emptyset$.
For $u_{0k}(x)=\frac{u_0(x_1+r_k x)}{r_k}$, 
by compactness argument as above, we have $u_{0k}\rightarrow u_{1}$ 
for another function  
$u_1$ 
and
\begin{eqnarray*}
 \int_{B_R}\I{u_1}\div\phi\longleftarrow\int_{B_R}\I{u_{0k}}\div \phi=r^{1-n}_k\int_{B_{Rr_k}(x_1)}\I{u_0}\div\phi\left(\frac{x-x_1}{r_k}\right).\\\nonumber
\end{eqnarray*}
On the other hand 
\begin{eqnarray*}
\left|r^{1-n}_k\int_{B_{Rr_k}(x_1)}\I{u_0}\div\phi\left(\frac{x-x_1}{r_k}\right)\right|
=
\left|r^{1-n}_k\int_{\partial_{\rm red }\{u_0>0\}\cap B_{Rr_k}(x_1)}\nu\cdot \phi\left(\frac{x-x_1}{r_k}\right)\right|\\
\leq
\sup |\phi | r_k^{1-n}\H^{n-1}(\partial_{\rm red }\{u_0>0\}\cap B_{Rr_k}(x_1))\longrightarrow 0\quad {\rm{as}}\quad k\rightarrow \infty.
\end{eqnarray*}

Hence we infer that $\I{u_1}$ is a function of bounded variation which is constant a.e. in $B_R$.
The positive Lebesgue density property of $\{u_1=0\}$, translated to $u_1$ 
through compactness as in step {\bf 1)}, and strong maximum principle for the solutions of $\L u=0$ (see Hypothesis {\bf(H1)})
demands $u_1$ to be zero. This is in contradiction with  2) in Definition \ref{weak-def} 
since by compactness as in step {\bf 1)} the non-degeneracy  translates  to $u_1$.
Thus $K_0=\emptyset$ and we obtain $\H^{n-1}(\partial\{u_0>0\}\setminus \partial_{\rm red}\{u_0>0\})=0$.

\smallskip

To show that $\partial_{\rm red } \{u_0>0\}$ is smooth, we notice 
that if $z_0\in\partial_{\rm red }\{u_0>0\}$ and $\nu(z_0)$
is the normal at $z_0$ in the sense of 4.5.5.  \cite{Federer}, 
then by the uniform Lebesgue density of $\{ u_0=0\}$ and nondegeneracy
of $u_0$ from 2) of Definition \ref{weak-def} we have 
$u_0\in F(\frac\sigma 2, 1; \infty)$ in $B_{2\rho}(z_0)$ in direction $e=\nu(z_0)$
for $\sigma\leq \sigma_0$ and $\rho\leq \tau_0\sigma^{\frac 2\beta}$.
Here $F$ is the flatness class
defined as in \cite{Wolan-adv} definition 9.1.

At this point we don't know if $u_0$ is a weak solution
 so we cannot immediately apply the ''flatness implies regularity`` to $u_0$.
However from (\ref{AHD}) we conclude that $u_k$, 
which is a weak solution, is in $F(\sigma, 1; \infty)$ in $B_\rho(z_0)$ in the
direction of $e$ for sufficiently large $k$ (this is because $z_0\in \fbr{u_0}$ and $e$
is the normal at $z_0$ in measure theoretic sense). 
{Therefore, the surfaces $\fb{u_k}$ are all graphs of $C^{1, \alpha}$ functions in direction $e$ with uniform 
bounds so that we get the same property for $\fb{u_0}$}
Thus $\partial\{u_k>0\}$ are $C^3$ smooth in $B_{\frac \rho 4}(z_0)$ in the direction of $e$
and this translates to $\partial \{u_0>0\}$ in $B_{\frac \rho 4}(z_0)$.

\medskip

{ \bf Step 3)}
Now we can finally show that $u_0$ satisfies the equation in 3) of Definition \ref{weak-def}.
Take a compactly supported smooth function $\zeta $ and fix a $\delta>0$ small. 
Let $\mathscr F_1$ be a finite subcovering of
$\supp\zeta \cap(\partial\{u_0>0\}\setminus \partial_{\rm red }\{u_0>0\})$ 
by balls $B_{t_i}(y_i)$ such that
$\displaystyle\sum_{i=1}^{N'(\delta)} t_i^{n-1}<\delta$, see step {\bf 2)}. 
Then using partial integration we get

\begin{eqnarray*}
 \int g(|\na u_0|)\frac{\na u_0}{|\na u_0|}\na \zeta&=&
\int_{\mathscr F_1} g(|\na u_0|)\frac{\na u_0}{|\na u_0|}\na \zeta+\int_{\supp\zeta\setminus \mathscr F_1} g(|\na u_0|)\frac{\na u_0}{|\na u_0|}\na \zeta\\\nonumber
&=&\int_{\mathscr F_1} g(|\na u_0|)\frac{\na u_0}{|\na u_0|}\na \zeta-\int_{\supp\zeta\setminus \mathscr F_1\cap \{u_0>0\}}\zeta\L u_0\\
&&+\int_{\fbr{u_0}\setminus \mathscr F_1}\zeta  \frac{g(|\na u_0|)}{|\na u_0|}\frac{\partial u_0}{\partial \nu}\\\nonumber
&=&o_\delta(1)-g(\ell)\int_{\partial(\supp\zeta\setminus\mathscr F_1)}\zeta=o_\delta(1)-
\lambda_0\int_{\partial(\supp\zeta\setminus\mathscr F_1)}\zeta.
\end{eqnarray*}
To get the last line we used the definition of $\lambda_0=g(\ell)$
and that by step {\bf 2)} $\fbr {u_0}$ is smooth, hence, the free boundary condition $|\na u_0|=\ell$
holds in the classical sense.
Sending $\delta\rightarrow 0$ we conclude that $u_0$ is a weak solution.
 \end{proof}


\end{document}